\documentclass[10pt]{amsart}

\usepackage[latin1]{inputenc}
\usepackage[english]{babel}
\usepackage{color}
\usepackage{indentfirst}
\usepackage{amssymb}
\usepackage{amsthm}

\newcommand{\sub}{\subseteq}

\newcommand{\Nat}{\mathbb{N}}

\def\cG{{{\mathcal G}}}
\def\cS{{{\mathcal S}}}
\def\cM{{{\mathcal M}}}
\def\N{\mathbb N}
\def\epsilon{\varepsilon}

\newtheorem{theo}{Theorem}[section]
\newtheorem{lem}[theo]{Lemma}
\newtheorem{pro}[theo]{Proposition}
\newtheorem{cor}[theo]{Corollary}
\newtheorem{defi}[theo]{Definition}
\newtheorem{fact}[theo]{Fact}
\newtheorem{rem}[theo]{Remark}

\newtheorem{problem}[theo]{Problem}

\numberwithin{equation}{section}

\title[Convex combinations of weak$^*$-convergent sequences]{Convex combinations of weak$^*$-convergent sequences and the Mackey topology}

\author{Antonio Avil\'{e}s}
\address{Departamento de Matem\'{a}ticas\\ Facultad de Matem\'{a}ticas\\ Universidad de Murcia\\ 30100 Espinardo (Murcia)\\ Spain} 
\email{avileslo@um.es}

\author{Jos\'{e} Rodr\'{i}guez}
\address{Departamento de Matem\'{a}tica Aplicada\\ Facultad de Inform\'{a}tica\\ Universidad de Murcia\\ 30100 Espinardo (Murcia)\\ Spain} 
\email{joserr@um.es}

\date{\today}

\subjclass[2000]{46B26,46B50}

\keywords{Mackey topology, convex block subsequence, property (K), strongly weakly compactly generated space, Grothendieck space}

\thanks{Research partially supported by {\em Ministerio de Econom\'{i}a y Competitividad} and {\em FEDER} (project MTM2014-54182-P).
This work was also partially supported by the research project 19275/PI/14 funded by {\em Fundaci\'{o}n S\'{e}neca - Agencia de Ciencia y Tecnolog\'{i}a 
de la Regi\'{o}n de Murcia} within the framework of {\em PCTIRM 2011-2014}.}

\begin{document}

\begin{abstract}
A Banach space $X$ is said to have property~(K) if every $w^*$-convergent sequence in~$X^*$
admits a convex block subsequence which converges with respect to the Mackey topology.
We study the connection of this property with strongly weakly compactly generated
Banach spaces and its stability under subspaces, quotients and $\ell^p$-sums. 
We extend a result of Frankiewicz and Plebanek by proving that property~(K) is preserved by $\ell^1$-sums 
of less than $\mathfrak{p}$ summands. Without any cardinality restriction, we show that 
property~(K) is stable under $\ell^p$-sums for $1<p<\infty$.
\end{abstract}

\maketitle

\section{Introduction}

A classical result by Mazur ensures that if $(x_n)$ is a weakly convergent sequence
in a Banach space, then there is a {\em convex block subsequence} of~$(x_n)$ which is norm convergent.
By a convex block subsequence we mean a sequence $(y_k)$ of the form
$$
	y_k=\sum_{n\in I_k}a_n x_n
$$
where $(I_k)$ is a sequence of finite subsets of~$\N$ with $\max(I_k) < \min(I_{k+1})$ and $(a_n)$ is a sequence
of non-negative real numbers such that $\sum_{n\in I_k}a_n=1$ for all~$k \in \N$. 
Mazur's theorem can be deduced from a general principle asserting that 
if $E$ is a locally convex space (with topological dual~$E'$) and $\mathcal{T}$ is any locally convex topology on~$E$ which is
compatible with the dual pair~$\langle E,E' \rangle$ (meaning that $(E,\mathcal{T})'=E'$), then
\begin{equation}\label{eqn:convex}
	\overline{C}=\overline{C}^{\, \mathcal{T}}
	\quad\mbox{for any convex set }C\sub E.
\end{equation}

If $X$ is a Banach space, then $(X^*,w^*)'=X$ and the strongest locally convex topology on~$X^*$
which is compatible with the dual pair $\langle X^*,X \rangle$ is the Mackey topology $\mu(X^*,X)$, i.e.
the topology of uniform convergence on weakly compact subsets of~$X$. When applied to this setting, \eqref{eqn:convex} reads as
\begin{equation}\label{eqn:convexMackey}
	\overline{C}^{\, w^*}=\overline{C}^{\, \mu(X^*,X)}
	\quad\mbox{for any convex set }C\sub X^*.
\end{equation}
If one assumes that $(B_{X^*},\mu(X^*,X))$ is metrizable, then~\eqref{eqn:convexMackey} allows to
conclude that any $w^*$-convergent sequence in~$X^*$ admits a convex block subsequence which is $\mu(X^*,X)$-convergent.
This is the keystone of our paper:

\begin{defi}\label{defi:PropertyK}
A Banach space~$X$ has {\em property~(K)} if every $w^*$-convergent sequence in~$X^*$
admits a convex block subsequence which is $\mu(X^*,X)$-convergent.
\end{defi}

Property~(K) was invented by Kwapie\'{n} to provide an alternative approach to some results of Kalton and Pe\l czy\'{n}ski~\cite{kal-pel} 
on subspaces of~$L^1[0,1]$. A weakening of property~(K) was used by Figiel, Johnson and Pe\l czy\'{n}ski~\cite{fig-alt2} 
to show that, in general, the separable complementation property is not inherited by subspaces. 
It is not difficult to check that 
the spaces $c_0$ and $C[0,1]$ fail property~(K). On the other hand, the basic examples of Banach spaces~$X$ having property~(K) are:
\begin{itemize}
\item {\em Schur spaces}, i.e. spaces such that weak convergent sequences in~$X$ are norm convergent. In this case, 
the topologies $w^*$ and $\mu(X^*,X)$ agree on bounded subsets of~$X^*$.
\item {\em Grothendieck spaces}, i.e. spaces such that $w^*$-convergent sequences in~$X^*$ are weakly convergent, like reflexive spaces and $\ell^\infty$.
In this case, Mazur's theorem ensures that $X$ has property~(K).
\item {\em Strongly weakly compactly generated (SWCG) spaces}, i.e.
those for which there exists a weakly compact set $G \sub X$ such that
for every weakly compact set $L \sub X$ and every $\varepsilon>0$ there is $n\in \N$ such that $L \sub nG + \varepsilon B_X$.
These are precisely the spaces for which $(B_{X^*},\mu(X^*,X))$ is metrizable,~\cite{sch-whe}. A typical
example is $L^1(\mu)$ for any finite measure~$\mu$.
\end{itemize}

In this paper we discuss some features of property~(K). We point out
that this property is equivalent to reflexivity for Banach spaces not containing subspaces isomorphic to~$\ell^1$
(Theorem~\ref{theo:no-ell1}). We discuss its stability under subspaces and quotients. The three-space problem
for property~(K) is considered and we prove that a Banach space~$X$ has property~(K) whenever
there is a subspace~$Y \sub X$ such that $Y$ is Grothendieck and $X/Y$ has property~(K) (Theorem~\ref{theo:3space}).
On the other hand, we extend the fact that SWCG Banach spaces
have property~(K) by proving that the same holds for Banach spaces which are strongly generated by less than~$\mathfrak{p}$
weakly compact sets (Theorem~\ref{theo:K-p}). Two different proofs of this result are given. One of them
is based on a diagonal argument by Haydon, Levy and Odell~\cite{hay-lev}, which is also
used to show that property~(K) is preserved by $\ell^1$-sums of less than $\mathfrak{p}$ summands
(Theorem~\ref{theo:l1sums}). This is an improvement of a result by Frankiewicz and Plebanek~\cite{fra-ple},
who proved the same statement under the additional assumption that each summand is either $L^1(\mu)$ for some finite measure~$\mu$
or a Grothendieck space. We stress that Pe\l czy\'{n}ski (unpublished, 1996) proved that the $\ell^1$-sum of $\mathfrak{c}$ copies of~$L^1[0,1]$ 
fails property~(K), see~\cite[Example~4.I]{fig-alt2}. As pointed out in~\cite{fra-ple}, a similar counterexample can 
also be constructed for the $\ell^1$-sum of $\mathfrak{b}$ copies of~$L^1[0,1]$. We finish the paper by proving that
property~(K) is preserved by arbitrary $\ell^p$-sums for any $1<p<\infty$ (Theorem~\ref{theo:lpsums}).

\subsection*{Terminology}

All our linear spaces are real. Given a sequence $(f_n)$ in a linear space, a {\em rational convex block subsequence} of~$(f_n)$
is a sequence $(g_k)$ of the form
$$
	g_k=\sum_{n\in I_k}a_n f_n
$$
where $(I_k)$ is a sequence of finite subsets of~$\N$ with $\max(I_k) < \min(I_{k+1})$ and $(a_n)$ is a sequence
of non-negative rational numbers such that $\sum_{n\in I_k}a_n=1$ for all~$k \in \N$. 
By a {\em subspace} of a Banach space we mean a closed linear subspace.
Given a Banach space~$Z$, its norm is denoted
by $\|\cdot\|_Z$ if needed explicitly. We write $B_Z=\{z\in Z: \|z\|_Z\leq 1\}$. The dual of~$Z$ is denoted by~$Z^*$
and the evaluation of $z^*\in Z^*$ at~$z\in Z$ is denoted by either $z^*(z)$ or $\langle z^*,z\rangle$.

We will use without explicit mention the following elementary facts: (i)~a 
Banach space~$X$ has property~(K) if and only if every $w^*$-convergent sequence in~$X^*$
admits a rational convex block subsequence which is $\mu(X^*,X)$-convergent; (ii)~in order to check 
whether a Banach space~$X$ has property~(K), it suffices to test on $w^*$-null sequences in~$B_{X^*}$ 
or $w^*$-null linearly independent sequences. 

The cardinality of a set $S$ is denoted by~$|S|$. 
The cardinal~$\mathfrak{p}$ is defined as the least cardinality of a family $\cM$ of infinite
subsets of~$\N$ having the following properties:
\begin{itemize}
\item $\bigcap \mathcal{N}$ is infinite for every finite subfamily $\mathcal{N} \sub \cM$.
\item There is no infinite set $A \sub \N$ such that $A \setminus M$ is finite for all $M\in \cM$.
\end{itemize}
It is known that $\omega_1\leq \mathfrak{p} \leq \mathfrak{c}$ and that
the equality $\mathfrak{p} =\mathfrak{c}$ holds under Martin's Axiom.
We refer the reader to~\cite{freMartin,dou2} for further information on~$\mathfrak{p}$.

\section{Results}\label{section:Results}

The dual ball $B_{X^*}$ of a Banach space~$X$ is said to be {\em $w^*$-convex block compact} if every 
sequence in~$B_{X^*}$ admits a $w^*$-convergent convex block subsequence. 
A result of Bourgain~\cite{bou:79} (cf. \cite[Proposition~11]{pfi-J} and~\cite{sch-3}) states that
$B_{X^*}$ is $w^*$-convex block compact whenever $X$ contains no subspace isomorphic to~$\ell^1$.
As an application we get the following:

\begin{theo}\label{theo:no-ell1}
Let $X$ be a Banach space not containing subspaces isomorphic to~$\ell^1$. Then 
$X$ is reflexive if and only if it has property~(K).
\end{theo}
\begin{proof} It only remains to prove the ``if'' part.
Assume that $X$ has property~(K). We will check that $B_{X^*}$ is weakly compact
by proving that any sequence in~$B_{X^*}$ admits a convex block subsequence which is norm convergent
(cf. \cite[Corollary~2.2]{die-alt2}). Let $(x_n^*)$ be a sequence in~$B_{X^*}$. Since $X$ contains no isomorphic copy of~$\ell^1$,
there is a $w^*$-convergent convex block subsequence $(y^*_k)$ of~$(x_n^*)$ (see the paragraph
preceding the theorem). By passing to a further convex block subsequence, not relabeled, 
we can assume that $(y^*_k)$ is $\mu(X^*,X)$-convergent. Then $(y^*_k)$ is norm convergent, because
the fact that $X$ contains no isomorphic copy of~$\ell^1$ is equivalent to saying that the norm topology 
and $\mu(X^*,X)$ agree sequentially on~$X^*$ (see e.g. \cite[Theorem~5.53]{fab-ultimo}). 
This shows that $B_{X^*}$ is weakly compact and so $X$ is reflexive.
\end{proof}

In general, property~(K) is not inherited by subspaces. For instance, $\ell^\infty$ has property~(K) 
(because it is Grothendieck) and $c_0$ does not (cf. \cite[Proposition~C]{kal-pel}). 
Indeed, let $(e_n^*)$ be the canonical basis of~$\ell^1=c_0^*$. Then $(e_n^*)$ is $w^*$-null
and does not admit any $\mu(\ell^1,c_0)$-null convex block subsequence. For if 
$y^*_k=\sum_{n\in I_k}a_n e_n^*$ is any convex block subsequence, then
$y^*_k(x_k)=1$ for every~$k\in \N$, where $(x_k)$ is the weakly null sequence in~$c_0$ defined by 
$x_k:=\sum_{n\in I_k}e_n$ (here $(e_n)$ denotes the canonical basis of~$c_0$).

Following~\cite{cas-gon-pap}, we say that a subspace~$Y$ of a Banach space~$X$ is {\em $w^*$-extensible} 
if every $w^*$-null sequence in~$Y^*$ admits a subsequence which can be extended to a $w^*$-null sequence in~$X^*$. 
Obviously, any complemented subspace is $w^*$-extensible. On the other hand, 
it is not difficult to check that if~$X$ is a Banach space such that $B_{X^*}$ is $w^*$-sequentially compact
(this holds, for instance, if $X$ is weakly compactly generated, see e.g. \cite[Theorem~13.20]{fab-ultimo}), 
then all subspaces of~$X$ are $w^*$-extensible (see \cite[Theorem~2.1]{wan-alt}). 

We omit the easy proof of the following result:

\begin{pro}\label{pro:extensible}
Let $X$ be a Banach space and $Y \sub X$ a $w^*$-extensible subspace. 
If $X$ has property~(K), then $Y$ has property~(K).
\end{pro}

\begin{cor}\label{cor:separableK}
Let $X$ be an SWCG Banach space. Then any subspace of~$X$ has property~(K).
\end{cor}
\begin{proof}
Since $X$ is weakly compactly generated, any subspace of~$X$ is $w^*$-extensible
(see the comments preceding Proposition~\ref{pro:extensible}). 
The result now follows from Proposition~\ref{pro:extensible}
and the fact that any SWCG Banach space has property~(K).
\end{proof}

In particular, it follows that any subspace of~$L^1(\mu)$ (for a finite measure~$\mu$) has property~(K). We stress that 
there exist non-SWCG subspaces of~$L^1[0,1]$, like the one constructed by Mercourakis and Stamati in~\cite[Section~3]{mer-sta-2}
(cf. \cite[Theorem~7.5]{avi-ple-rod-5}).

\begin{cor}\label{cor:K-no-c0}
Let $X$ be a Banach space having property~(K). Then: 
\begin{enumerate}
\item[(i)] $X$ contains no complemented subspace isomorphic to~$c_0$.
\item[(ii)] $X$ contains no subspace isomorphic to~$c_0$ if $X$ has the separable complementation property.  
\end{enumerate}
\end{cor}
\begin{proof}
(i) follows from Proposition~\ref{pro:extensible} and the failure of property~(K) in~$c_0$. (ii)
is a consequence of~(i) and Sobczyk's theorem (see e.g. \cite[Theorem~5.11]{fab-ultimo}).
\end{proof}

\begin{cor}\label{cor:GrothendieckCK}
Let $L$ be a compact Hausdorff topological space. Then $C(L)$
is Grothendieck if and only if it has property~(K).
\end{cor}
\begin{proof}
Combine Corollary~\ref{cor:K-no-c0}(i) and the fact that $C(L)$ is Grothendieck if and only if it does not 
contain complemented subspaces isomorphic to~$c_0$ (see e.g. \cite[p.~74]{cem-men}).
\end{proof}

We next discuss the behavior of property~(K) with respect to quotients.

\begin{theo}\label{theo:3space}
Let $X$ be a Banach space and $Y \sub X$ a subspace. 
\begin{enumerate}
\item[(i)] If $Y$ is reflexive and $X$ has property~(K), then $X/Y$
has property~(K).
\item[(ii)] If $Y$ is Grothendieck and $X/Y$ has property~(K), then $X$ has property~(K).
\end{enumerate}
\end{theo}
\begin{proof} Let $q:X \to X/Y$ be the quotient map. Its 
adjoint $q^*: (X/Y)^* \to X^*$ is an isomorphism (and $w^*$-$w^*$-homeomorphism) onto~$Y^\perp$.

(i) Let $(\xi_n^*)$ be a $w^*$-null sequence in $(X/Y)^*$. Then 
$(q^*(\xi_n^*))$ is $w^*$-null in~$X^*$. Since $X$ has property~(K), there is
a convex block subsequence $(\eta_n^*)$ of $(\xi_n^*)$ such that $(q^*(\eta_n^*))$ is
$\mu(X^*,X)$-null. We claim that $(\eta_n^*)$ converges  
to~$0$ with respect to~$\mu((X/Y)^*,X/Y)$. Indeed, take any weakly compact set $L \sub X/Y$. 
Since $q$ is open and $L$ is bounded and weakly closed, 
we can find a bounded and weakly closed set $L_0 \sub X$ such that $q(L_0)=L$.
Since $Y$ is reflexive and $L$ is weakly compact, $L_0$ is weakly compact as well (see e.g. \cite[2.4.b]{cas-gon}).
Then
$$
	\lim_{n\to \infty} \, \sup_{z\in L}\big|\langle \eta_n^*,z \rangle\big|= \lim_{n\to \infty} \, 
	\sup_{x\in L_0}\big|\langle q^*(\eta_n^*),x \rangle\big|=0.
$$
This shows that $(\eta_n^*)$ converges to~$0$ with respect to~$\mu((X/Y)^*,X/Y)$.

(ii) Let $(x_n^*)$ be a $w^*$-null sequence in~$X$. Then the sequence of restrictions $(x_n^*|_Y)$
is $w^*$-null in~$Y^*$. Since $Y$ is Grothendieck, there is a convex block subsequence $(w_n^*)$ of~$(x^*_n)$ 
such that $\|w_n^*|_Y\|_{Y^*}\to 0$. Bearing in mind that the map
$$
	\phi: X^*/Y^{\perp} \to Y^*, \quad
	\phi(x^*+Y^\perp):= x^*|_{Y},
$$ 
is an isomorphism, we can find a sequence $(v_n^*)$ in~$Y^{\perp}$ such that 
\begin{equation}\label{eqn:norm}
	\|w_n^*+v_n^*\|_{X^*}\to 0. 
\end{equation}
Note that $(v_n^*)$ is $w^*$-null (by~\eqref{eqn:norm} and the fact that $(w_n^*)$ is $w^*$-null). Then
there is a $w^*$-null sequence $(\xi_n^*)$ in~$(X/Y)^*$ such that $q^*(\xi_n^*)=v_n^*$ for all $n\in \N$. 
Since $X/Y$ has~property~(K), there is a convex block subsequence $(\eta_n^*)$ of~$(\xi_n^*)$ 
converging to~$0$ with respect to $\mu((X/Y)^*,X/Y)$. It is now easy to check that
$(q^*(\eta_n^*))$ is a convex block subsequence of~$(v_n^*)$ converging to~$0$ with respect to~$\mu(X^*,X)$. 
By~\eqref{eqn:norm}, there is a convex block subsequence $(\tilde{w}_n^*)$ of~$(w_n^*)$ (and so of~$(x_n^*)$)
such that 
$$
	\|\tilde{w}_n^*+q^*(\eta_n^*)\|_{X^*}\to 0.
$$
Since $(q^*(\eta^*_n))$ is~$\mu(X^*,X)$-null, the same holds for $(\tilde{w}_n^*)$.
\end{proof}

We say that a family $\mathcal{G}$ of weakly compact subsets of a Banach space~$X$ {\em strongly generates}~$X$
if for every weakly compact set $L \sub X$ and every $\varepsilon>0$ there is $G \in \mathcal{G}$ such that $L \sub G + \varepsilon B_X$.
We denote by $SG(X)$ the least cardinality of a family of weakly compact subsets of~$X$ which strongly generates~$X$.
The cardinal $SG(X)$ coincides with the cofinality of an ordered structure studied in~\cite{avi-ple-rod-5}. Note that $X$ is SWCG
if and only if $SG(X)=\omega$ (see \cite[Theorem~2.1]{sch-whe}).

\begin{theo}\label{theo:K-p}
Let $X$ be a Banach space. If $SG(X)<\mathfrak{p}$, then $X$ has property~(K).
\end{theo}

We will give two different proofs of Theorem~\ref{theo:K-p}. The first one is
based on duality arguments inspired by the proof of \cite[Theorem~2.1]{sch-whe}. 

Given a topological space~$T$, the {\em character} of a point~$t\in T$, denoted by 
$\chi(t,T)$, is the least cardinality of a neighborhood basis of~$t$.

\begin{lem}\label{pro:SGcharacter}
Let $X$ be a Banach space. Then $SG(X)=\chi(0,(B_{X^*},\mu(X^*,X)))$.
\end{lem}
\begin{proof}
For every weakly compact set $L \sub X$ and $\epsilon>0$, we consider
the neighborhood of~$0$ in~$(B_{X^*},\mu(X^*,X))$ defined by
$$
	V(L,\epsilon):=\{x^*\in B_{X^*}: \, P_{L}(x^*)<\epsilon\},
$$
where $p_L(x^*):=\sup\{|x^*(x)|: x\in L\}$. The collection
of all sets of the form $V(L,\epsilon)$ is a basis of neighborhoods of~$0$ in~$(B_{X^*},\mu(X^*,X))$.

We first prove that $SG(X)\geq \chi(0,(B_{X^*},\mu(X^*,X)))$. Fix a family $\mathcal{G}$ of weakly compact
sets strongly generating~$X$ such that $|\mathcal{G}|=SG(X)$.
{\em We claim that} 
$$
	\mathcal{V}:=\Big\{V\Big(G,\frac{1}{n}\Big): \, G\in \mathcal{G}, \, n\in \Nat\Big\}
$$
{\em is a basis of neighborhoods of~$0$ in~$(B_{X^*},\mu(X^*,X))$.} Indeed, take any
weakly compact set $L \sub X$ and $\epsilon>0$. Then there is $G \in \mathcal{G}$ such that $L \sub G+\frac{\epsilon}{2} B_X$, so that
$V(G,\frac{1}{n})\sub V(L,\epsilon)$ whenever $\frac{1}{n} \leq \frac{\epsilon}{2}$. This proves the claim and shows that
$\chi(0,(B_{X^*},\mu(X^*,X)))\leq |\mathcal{V}|\leq |\mathcal{G}|=SG(X)$.

We now prove that $SG(X)\leq \chi(0,(B_{X^*},\mu(X^*,X)))$. Let $\mathcal{W}$ be a
basis of neighborhoods of~$0$ in~$(B_{X^*},\mu(X^*,X))$ with $|\mathcal{W}|=\chi(0,(B_{X^*},\mu(X^*,X)))$. 
For each $W \in \mathcal{W}$ we can choose a weakly compact set $L_W\sub X$ and $\epsilon_W>0$
such that $V(L_W,\epsilon_W) \sub W$. By the Krein-Smulyan theorem
(see e.g. \cite[p.~51, Theorem~11]{die-uhl-J}), we can assume further
that each $L_W$ is absolutely convex. {\em We claim that}
$$
	\mathcal{G}:=\{nL_W: \, W\in \mathcal{W}, \, n\in \Nat\}
$$
{\em strongly generates~$X$.} To this end, fix a weakly compact set $L\sub X$ and $\epsilon>0$.
There is some $W\in \mathcal{W}$ such that $V(L_W,\epsilon_W) \sub V(L,\epsilon)$. Pick
$n\in \Nat$ with $n\geq \frac{\epsilon}{\epsilon_W}$. We will check that 
\begin{equation}\label{equation:inclusion}
	L \sub n L_W + \epsilon B_X. 
\end{equation}
By contradiction, suppose there is $x\in L$ which does not belong to the convex closed set $nL_W + \epsilon B_X$. 
By the Hahn-Banach separation theorem, there is $x^*\in X^*$ with $\|x^*\|_{X^*}=1$ such that
\begin{equation}\label{equation:HB}
	x^*(x)>\sup\{x^*(y): \, y\in nL_W + \epsilon B_X\}=np_{L_W}(x^*)+\epsilon.
\end{equation}
Set $c:=(\frac{1}{\epsilon_W}p_{L_W}(x^*)+1)^{-1}$. Then $cx^*\in B_{X^*}$ and 
$cx^*\in V(L_W,\epsilon_W) \sub V(L,\epsilon)$. In particular,
$cx^*(x)<\epsilon$ and therefore
$$
	x^*(x)<\frac{\epsilon}{\epsilon_W} p_{L_W}(x^*)+\epsilon \leq n p_{L_W}(x^*)+\epsilon,
$$ 
which contradicts~\eqref{equation:HB}. Hence~\eqref{equation:inclusion} holds and this completes the proof
that $\mathcal{G}$ strongly generates~$X$, as claimed. It follows that
$$
	SG(X)\leq |\mathcal{G}| \leq |\mathcal{W}|=\chi(0,(B_{X^*},\tau(X^*,X)))
$$
and the proof of the lemma is finished.
\end{proof}

\begin{proof}[First proof of Theorem~\ref{theo:K-p}]
Let $(x_n^*)$ be a $w^*$-null sequence in~$B_{X^*}$. Fix $k\in \Nat$. Define
$A_k:={\rm co}_{\mathbb{Q}}\{x^*_n:n \geq k\} \sub B_{X^*}$ and observe that
$$
	0 \in \overline{\{x^*_n: \, n \geq k\}}^{w^*} \sub \overline{{\rm co}\{x^*_n: \, n \geq k\}}^{w^*}\stackrel{\eqref{eqn:convexMackey}}{=}
	\overline{{\rm co}\{x^*_n: \, n \geq k\}}^{\mu(X^*,X)}=
	\overline{A_k}^{\mu(X^*,X)}.
$$ 
Since $SG(X)=\chi(0,(B_{X^*},\mu(X^*,X)))$ (Lemma~\ref{pro:SGcharacter})
is strictly less than~$\mathfrak{p}$ and $A_k$ is countable, there is a sequence $(y^*_{n,k})_{n\in \N}$ in~$A_k$ converging
to~$0$ with respect to~$\mu(X^*,X)$ (see e.g. \cite[24A]{freMartin}).

Now, using that $\chi(0,(B_{X^*},\mu(X^*,X))) < \mathfrak{p}\leq \mathfrak{b}$ (see e.g. \cite[14B]{freMartin}), 
we can find a map $\varphi: \Nat \to \Nat$ such that the sequence $(y^*_{\varphi(k),k})$ converges to~$0$
with respect to~$\mu(X^*,X)$ (cf. \cite[Lemma~1]{fra-ple}). 
Finally, notice that $y^*_{\varphi(k),k}\in A_k$ for all $k\in \Nat$
and, therefore, some subsequence of~$(y^*_{\varphi(k),k})$ is a convex block subsequence of~$(x_n^*)$.
This proves that $X$ has property~(K).
\end{proof}

Our second proof of Theorem~\ref{theo:K-p} is based on Lemma~\ref{lem:HLO} below, which is taken from~\cite[2A]{hay-lev}.
Note that the statement given here is slightly different from the original one 
(it follows from the proof of~\cite[2A]{hay-lev}).

Given two sequences $(g_n)$ and $(h_n)$ in a linear space, we write 
$$
	(g_n) \preceq (h_n)
$$ 
to denote that $(g_n)$ is {\em eventually} a rational convex block subsequence of~$(h_n)$, i.e.
there is $n_0\in \N$ such that $(g_{n})_{n\geq n_0}$ is a rational convex block subsequence of~$(h_n)$.

\begin{lem}[Haydon-Levy-Odell]\label{lem:HLO}
Let $\kappa$ be a cardinal with $\kappa<\mathfrak{p}$.
Let $(f_n)$ be a linearly independent sequence in a linear space and, 
for each $\alpha<\kappa$, let $(f_{n,\alpha})$ be a rational convex block subsequence of~$(f_n)$.
Suppose that for every finite set $F \sub \kappa$ there is a rational convex block subsequence $(g_n)$ of~$(f_n)$
such that $(g_n) \preceq (f_{n,\alpha})$ for all $\alpha\in F$. Then there is 
a rational convex block subsequence $(h_n)$ of~$(f_n)$
such that $(h_n) \preceq (f_{n,\alpha})$ for all $\alpha<\kappa$.
\end{lem}

\begin{lem}\label{lem:corHLO}
Let $\kappa$ be a cardinal with $\kappa<\mathfrak{p}$.
Let $\mathcal{S}$ be a set of linearly independent sequences in a linear space.
For each $\alpha<\kappa$, let $\mathcal{S}_\alpha \sub \cS$ such that: 
\begin{enumerate}
\item[(i)] If $(g_n)\in \cS_\alpha$ and $(h_n) \preceq (g_n)$, then $(h_n)\in \cS_\alpha$.
\item[(ii)] Every sequence of~$\mathcal{S}$ admits a rational convex block subsequence belonging to~$\cS_\alpha$.
\end{enumerate}
Then every sequence of~$\cS$ admits a rational convex block subsequence belonging to~$\bigcap_{\alpha<\kappa} \cS_\alpha$.
\end{lem}
\begin{proof}
Fix $(f_n)\in \cS$.  We will prove 
the existence of a rational convex block subsequence of~$(f_n)$ belonging to~$\bigcap_{\alpha<\kappa} \cS_\alpha$.

{\sc Step 1.} We will construct by transfinite induction a collection $\{(f_{n,\alpha}):\alpha<\kappa\}$ of rational
convex block subsequences of~$(f_n)$ with the following properties:
\begin{enumerate}
\item[($P_\alpha$)] $(f_{n,\alpha}) \in \cS_\alpha$ for all $\alpha<\kappa$.
\item[($Q_{\alpha,\beta}$)] $(f_{n,\beta})\preceq (f_{n,\alpha})$ whenever $\alpha\leq\beta <\kappa$.
\end{enumerate}
For $\alpha=0$, we just use~(ii) to select any rational convex block subsequence of~$(f_{n})$ belonging to~$\cS_0$.
Suppose now that $\gamma<\kappa$ and that we have already constructed 
a collection $\{(f_{n,\alpha}):\alpha<\gamma\}$ of rational convex block subsequences of~$(f_n)$
such that ($P_\alpha$) and ($Q_{\alpha,\beta}$) hold whenever $\alpha\leq\beta<\gamma$. 
Since $|\gamma| < \mathfrak{p}$, Lemma~\ref{lem:HLO} ensures the existence
of a rational convex block subsequence~$(f_{n,\gamma})$ of~$(f_{n})$ 
such that $(f_{n,\gamma})\preceq (f_{n,\alpha})$ for all $\alpha<\gamma$.
Property~(i) yields $(f_{n,\gamma})\in \bigcap_{\alpha<\gamma}\cS_\alpha \sub \cS$. Now property~(ii)
implies that, by passing to a further rational convex block subsequence, we can assume
that ($P_\gamma$) holds. Clearly, ($Q_{\alpha,\beta}$) also holds for every $\alpha\leq \beta \leq \gamma$.

{\sc Step 2.} Since $\kappa<\mathfrak{p}$, another appeal to Lemma~\ref{lem:HLO} allows
us to extract a rational convex block subsequence $(h_n)$ of~$(f_n)$
such that $(h_n) \preceq (f_{n,\alpha})$ for every $\alpha<\kappa$. In particular,
$(h_n)\in \cS_\alpha$ for all $\alpha<\kappa$ (by~(i)).
\end{proof}

\begin{lem}\label{lem:DFJP}
Let $X$ be a Banach space and $L \sub X$ a weakly compact set. If $(x_n^*)$ is a $w^*$-convergent
sequence in~$X^*$, then there is a convex block subsequence of~$(x_n^*)$ which converges uniformly on~$L$.
\end{lem}
\begin{proof}
We can suppose without loss of generality that $(x_n^*)$ is $w^*$-null.
By the Davis-Figiel-Johnson-Pe\l czy\'{n}ski factorization theorem (see e.g. \cite[Theorem~13.22]{fab-ultimo}),
there exist a reflexive Banach space~$Y$ and a linear continuous map~$T:Y\to X$ in such a way that $L \sub T(B_Y)$.
Since $T^*:X^*\to Y^*$ is $w^*$-$w^*$-continuous, $(T^*(x_n^*))$ is $w^*$-null in~$Y^*$.
Since $Y$ is reflexive, there is a convex block subsequence~$(z_n^*)$ of~$(x_n^*)$ such that
$\|T^*(z_n^*)\|_{Y^*}\to 0$. Bearing in mind that
$$
	\sup_{x\in L}|z_n^*(x)| \leq \sup_{y\in B_Y}|z_n^*(T(y))|=\|T^*(z_n^*)\|_{Y^*}
	\quad\mbox{for all }n\in \N,
$$
we conclude that $(z_n^*)$ converges to~$0$ uniformly on~$L$.
\end{proof}

\begin{rem}\label{rem:K-stronggeneration}
Let $X$ be a Banach space and $\cG$ a family of weakly compact subsets of~$X$ which strongly generates~$X$. Then
a bounded sequence $(x_n^*)$ in~$X^*$ is $\mu(X^*,X)$-null if and only if it
converges to~$0$ uniformly on any element of~$\cG$.
\end{rem}
\begin{proof} We only have to check the ``if'' part. 
To this end, assume without loss of generality that $x_n^* \in B_{X^*}$ for all $n\in \N$. 
Take any weakly compact set $L\sub X$ and $\epsilon>0$. Then there is
$G\in \cG$ such that $L \sub G+\epsilon B_X$. Therefore
$$
	\sup_{x\in L}|x_n^*(x)| \leq
	\sup_{x\in G}|x_n^*(x)|+\epsilon \leq 2\epsilon
$$
for large enough~$n$.
\end{proof}

\begin{proof}[Second proof of Theorem~\ref{theo:K-p}]
Write $\kappa:=SG(X)<\mathfrak{p}$. Let $\mathcal{G}$ be a family of weakly
compact subsets of~$X$ which strongly generates~$X$ and such that $|\mathcal{G}|=\kappa$. 
Enumerate $\mathcal{G}=\{G_\alpha:\alpha<\kappa\}$.
Let $\cS$ be the set of all $w^*$-null linearly independent sequences in~$X^*$ and,
for each $\alpha<\kappa$, let $\cS_\alpha \sub \cS$ be the set consisting of all sequences of~$\cS$
which converge uniformly on~$G_\alpha$. Conditions~(i) and~(ii) of Lemma~\ref{lem:corHLO} are fulfilled 
for this choice of~$\cS$ and~$\cS_\alpha$ (note that (ii) follows from Lemma~\ref{lem:DFJP}). 

Now, take any $(x_n^*)\in \cS$. By Lemma~\ref{lem:corHLO}, $(x_n^*)$ 
admits a convex block subsequence~$(z_n^*)$ converging uniformly on each~$G_\alpha$. 
Since $\{G_\alpha:\alpha<\kappa\}$ strongly generates~$X$, 
the sequence $(z_n^*)$ is $\mu(X^*,X)$-null (Remark~\ref{rem:K-stronggeneration}).
This proves that $X$ has property~(K).
\end{proof}

To deal with Theorem~\ref{theo:l1sums} below
we need the following fact, which is folklore and can be deduced, for instance, 
by an argument similar to that of \cite[p.~104, Theorem~4]{die-uhl-J} (cf. \cite[Lemma~7.2]{kac-alt}).

\begin{fact}\label{fact:ProductTopology}
Let $X:=(\bigoplus_{i\in I}X_i)_{\ell^1}$ be the $\ell^1$-sum of
a collection $\{X_i\}_{i\in I}$ of Banach spaces. For each $i\in I$, let
$\pi_i:X \to X_i$ be the $i$th-coordinate projection. Then
for every weakly compact set $L \sub X$ and $\epsilon>0$ there is a finite set $I_0 \sub I$ such that
$$
	\sum_{i \in I \setminus I_0}\|\pi_i(x)\|_{X_i} \leq \epsilon \quad
	\mbox{for all }x\in L.
$$
\end{fact}

\begin{theo}\label{theo:l1sums}
Let $\{X_i\}_{i\in I}$ be a collection of Banach spaces having property~(K). If $|I|<\mathfrak{p}$,
then $(\bigoplus_{i\in I}X_i)_{\ell^1}$ has property~(K).
\end{theo}
\begin{proof} For notational convenience, we assume that the index set~$I$ is a cardinal, say~$\kappa$.
Write $X:=(\bigoplus_{\alpha<\kappa}X_\alpha)_{\ell^1}$ and identify
$X^*=(\bigoplus_{\alpha<\kappa}X^*_\alpha)_{\ell^\infty}$. Denote by 
$\pi_\alpha: X \to X_\alpha$ and $\rho_\alpha: X^* \to X^*_\alpha$
the $\alpha$th-coordinate projections for every $\alpha<\kappa$. Let 
$\mathcal{G}$ be the family of all weakly compact sets $G \sub X$ for which there is a finite
set $I_G \sub \kappa$ such that $\pi_\alpha(G)=\{0\}$ for every $\alpha\in \kappa\setminus I_G$.
From Fact~\ref{fact:ProductTopology} it follows that $\cG$ strongly generates~$X$.
Therefore, in order to prove that $X$ has property~(K)
it suffices to check that any $w^*$-null linearly independent sequence in~$X^*$
admits a convex block subsequence which converges uniformly on any element of~$\cG$
(Remark~\ref{rem:K-stronggeneration}). 

Let $\cS$ be the set of all $w^*$-null linearly independent sequences in~$X^*$ and,
for each $\alpha<\kappa$, let $\cS_\alpha \sub \cS$ be the set of all $(x_n^*)\in \cS$ such that
$(\rho_\alpha(x_n^*))$ is~$\mu(X_\alpha^*,X_\alpha)$-null. Conditions~(i) and~(ii) of Lemma~\ref{lem:corHLO} are satisfied
for this choice ((i) is immediate and (ii) holds because $\rho_\alpha$ is $w^*$-$w^*$-continuous
and $X_\alpha$ has property~(K)).
Therefore, every $(x_n^*) \in \cS$ admits a convex block subsequence $(y_n^*)$
such that $(\rho_\alpha(y_n^*))$ is $\mu(X_\alpha^*,X_\alpha)$-null for every $\alpha<\kappa$, which clearly implies that
$(y_n^*)$ converges to~$0$ uniformly on each element of~$\cG$. The proof is complete.
\end{proof}

On the other hand, property~(K) is stable under $\ell^p$-sums whenever $1<p<\infty$, without any restriction
on the cardinality of the family:

\begin{theo}\label{theo:lpsums}
Let $\{X_i\}_{i\in I}$ be a collection of Banach spaces having property~(K).
Then $(\bigoplus_{i\in I}X_i)_{\ell^p}$ has property~(K) for any $1<p<\infty$.
\end{theo}
\begin{proof}
Write $X:=(\bigoplus_{i\in I}X_i)_{\ell^p}$ and identify $X^*=(\bigoplus_{i\in I}X^*_i)_{\ell^q}$, where $1<q<\infty$ satisfies
$\frac{1}{p}+\frac{1}{q}=1$. Since every element of~$X^*$ is countably supported, we can assume without loss of generality
that $I$ is {\em countable}. Let $(x_n^*)$ be a $w^*$-null sequence in~$X^*$, that is, $(x_n^*)$ is bounded and for every
$i\in I$ the sequence $(\rho_i(x_n^*))$ is $w^*$-null in~$X_i^*$, where $\rho_i: X^*\to X_i^*$ denotes the $i$th-coordinate projection.
Assume without loss of generality that $(x_n^*)$ is linearly independent.
Since any $w^*$-null sequence in~$X_i^*$ admits a rational convex block subsequence which is $\mu(X_i^*,X_i)$-null
(because $X_i$ has property~(K)), an appeal to Lemma~\ref{lem:corHLO} (in the countable case) allows us to extract a convex block subsequence
of~$(x_n^*)$, not relabeled, such that for every $i\in I$ the sequence $(\rho_i(x_n^*))$ is $\mu(X_i^*,X_i)$-null.

Define $z_n:=(\|\rho_i(x_n^*)\|_{X_i^*})_{i\in I}\in \ell^q(I)$ for all $n\in \Nat$, so that $\|z_n\|_{\ell^q(I)}=\|x_n^*\|_{X^*}$.
Since $(z_n)$ is a bounded sequence in the reflexive space~$\ell^q(I)$, it admits a norm convergent convex block subsequence, that is,
there exist $z\in \ell^q(I)$, a sequence $(I_k)$ of finite subsets of~$\N$ with $\max(I_k) < \min(I_{k+1})$ and a sequence $(a_n)$
of non-negative real numbers with $\sum_{n\in I_k}a_n=1$ for all $k\in \N$, in such a way that  
$$
	\lim_{k\to \infty}\Big\|\sum_{n\in I_k}a_n z_n-z\Big\|_{\ell^q(I)} = 0.
$$
Define a convex block subsequence $(y_k^*)$ of~$(x_n^*)$ by
$$
	y_k^*:=\sum_{n\in I_k}a_n x_n^* \quad \mbox{for all }k\in \N.
$$
We will check that $(y_k^*)$ is $\mu(X^*,X)$-null.

To this end, fix any weakly compact set $L \sub X$ and $\epsilon>0$. We can assume that $L \sub B_X$.
There is a finite set $J \sub I$ such that
\begin{equation}\label{eqn:cola-z}
	\Big(\sum_{i\in I\setminus J} |\tilde{\rho}_i(z)|^q\Big)^{\frac{1}{q}} \leq \epsilon,
\end{equation}
where $\tilde{\rho}_i$ denotes the $i$th-coordinate functional on~$\ell^q(I)$. Choose $k_0\in \Nat$ 
large enough such that 
$$
	\Big\|\sum_{n\in I_k}a_n z_n-z\Big\|_{\ell^q(I)}\leq \epsilon
	\quad\mbox{for all }k> k_0.
$$
Then~\eqref{eqn:cola-z} yields
\begin{equation}\label{eqn:cola-w}
	\sup_{k> k_0} \Big(\sum_{i\in I\setminus J} \Big|\tilde{\rho}_i\Big(\sum_{n\in I_k}a_n z_n\Big)\Big|^q\Big)^{\frac{1}{q}} \leq 2\epsilon.
\end{equation}
Bearing in mind that 
$$
	\|\rho_i(y_k^*)\|_{X_i^*} \leq
	\sum_{n\in I_k}a_n\|\rho_i(x_n^*)\|_{X_i^*}=\tilde{\rho}_i\Big(\sum_{n\in I_k} a_n z_n\Big)
$$
for every $i\in I$ and $k\in \Nat$, from~\eqref{eqn:cola-w} we conclude that
$$	
	\sup_{k> k_0} \Big(\sum_{i\in I\setminus J} \|\rho_i(y_k^*)\|_{X_i^*}^q\Big)^{\frac{1}{q}} \leq 2\epsilon.
$$
H\"{o}lder's inequality applied to $(\|\rho_i(y_k^*)\|_{X_i^*})_{i\in I}\in \ell^q(I)$ now yields
\begin{equation}\label{eqn:cola-ww}
	\sup_{k> k_0} \, \sum_{i\in I \setminus J} |b_i| \cdot \|\rho_i(y_k^*)\|_{X_i^*} \leq  2\epsilon
	\quad \mbox{for every }(b_i)_{i\in I}\in B_{\ell^p(I)}.
\end{equation}

For each $i\in I$, the sequence $(\rho_i(x_n^*))$ is $\mu(X_i^*,X_i)$-null and so the same holds for its convex block subsequence $(\rho_i(y_k^*))$.
In particular, $(\rho_i(y_k^*))$ converges to~$0$ uniformly on the weakly compact set $\pi_i(L) \sub X_i$,
where $\pi_i:X \to X_i$ denotes the $i$th-coordinate projection. Since $J$ is finite,
we can find $k_1>k_0$ such that
\begin{equation}\label{eqn:yk-coordinates}
	 \big|\langle \rho_i(y_k^*),\pi_i(x) \rangle \big| \leq \frac{\epsilon}{|J|}
	 \quad\mbox{for every }k>k_1, \, i\in J \mbox{ and }x\in L.
\end{equation}
By putting together~\eqref{eqn:cola-ww} and~\eqref{eqn:yk-coordinates}, 
for every $k>k_1$ and $x\in L \sub B_X$ we get
\begin{multline*}
	\big|\langle y_k^*,x \rangle \big| \leq
	\sum_{i\in I} \big|\langle \rho_i(y_k^*),\pi_i(x) \rangle \big|\leq 
	\epsilon + \sum_{i\in I\setminus J} \big|\langle \rho_i(y_k^*),\pi_i(x) \rangle \big| \\
	\leq \epsilon + \sum_{i\in I\setminus J} \|\pi_i(x)\|_{X_i} \cdot \|\rho_i(y_k^*)\|_{X_i^*}  \leq
	3\epsilon.
\end{multline*}
This shows that $(y_k^*)$ is $\mu(X^*,X)$-null and the proof is finished.
\end{proof}

From the previous theorem it follows that $\ell^p(\ell^1)$ has property~(K) whenever $1<p<\infty$. However,
this space does not embed isomorphically into any SWCG Banach space, see \cite[Corollary~2.29]{kam-mer2}.

We finish the paper with some open questions:

\begin{problem}
Given a Banach space~$X$ and a subspace~$Y \sub X$,
does $X$ have property~(K) if both $Y$ and $X/Y$ have property~(K)?
\end{problem}

\begin{problem}
What is the ``optimal'' cardinal for Theorem~\ref{theo:l1sums} to work? 
\end{problem}

As we already mentioned in the introduction, the conclusion of Theorem~\ref{theo:l1sums} might fail
if $|I|=\mathfrak{b}$.

\begin{problem}
Let $X$ be a Banach space with unconditional basis not containing subspaces isomorphic to~$c_0$. 
Does $X$ have property~(K)? 
\end{problem}

A natural candidate to test the previous question is the Banach space $E_{1,2}(\mathcal{A})$ associated to
the adequate family~$\mathcal{A}$ of all chains of the dyadic tree, see e.g.~\cite{arg-mer}. This
space is not SWCG, see \cite[Example~2.6]{sch-whe} (cf. \cite[Example~2.9]{mer-sta-2}).

\subsection*{Acknowledgement}
The authors would like to thank G.~Plebanek for valuable discussions.

\providecommand{\MR}{\relax\ifhmode\unskip\space\fi MR }
\providecommand{\MRhref}[2]{%
  \href{http://www.ams.org/mathscinet-getitem?mr=#1}{#2}
}
\providecommand{\href}[2]{#2}

\bibliographystyle{amsplain}

\end{document}